\documentclass[a4paper, 10pt]{article}
\usepackage{amsmath}
\usepackage{amssymb,esint}
\usepackage{amscd}
\usepackage{xspace}
\usepackage{fancyhdr}
\newtheorem{theorem}{Theorem}[section]

\newtheorem{definition}[theorem]{Definition}

\newtheorem{lemma}[theorem]{Lemma}

\newtheorem{proposition}[theorem]{Proposition}
\newtheorem{remark}[theorem]{Remark}

\newenvironment{proof}[1][Proof]{\textbf{#1.} }{\hfill\rule{0.5em}{0.5em}}
{\catcode`\@=11\global\let\AddToReset=\@addtoreset
	\AddToReset{equation}{section}
	
	\AddToReset{theorem}{section}

\begin{document}
\title{An elliptic semilinear equation with source term and boundary measure data: the supercritical case
}
\author{
{\bf Marie-Fran\c{c}oise Bidaut-V\'eron\thanks{E-mail address: veronmf@univ-tours.fr}}\\[0.5mm]
 {\bf Giang Hoang\thanks{ E-mail address: hgiangbk@gmail.com }}\\[0.5mm]
{\bf Quoc-Hung Nguyen\thanks{ E-mail address: quoc-hung.nguyen@epfl.ch
}}\\[0.5mm]
{\bf Laurent V\'eron\thanks{ E-mail address: Laurent.Veron@lmpt.univ-tours.fr}}\\[2mm]
{\small Laboratoire de Math\'ematiques et Physique Th\'eorique, }\\
{\small  Universit\'e Fran\c{c}ois Rabelais,  Tours,  FRANCE}}
\date{}  

\maketitle
\begin{abstract} We give new criteria for the existence of weak solutions to an equation with a super linear source term 
\begin{align*}
-\Delta u = u^q ~~\text{in}~\Omega,~~u=\sigma~~\text{on }~\partial\Omega
\end{align*}
where $\Omega$ is a either a bounded smooth domain or $\mathbb{R}_+^{N}$, $q>1$  and $\sigma\in \mathfrak{M}^+(\partial\Omega)$ is a nonnegative Radon measure on $\partial\Omega$.  One of the criteria we obtain is expressed in terms of some Bessel capacities on $\partial\Omega$. We also give a sufficient condition for the existence of weak solutions to equation with source mixed terms.
 \begin{align*}
 -\Delta u = |u|^{q_1-1}u|\nabla u|^{q_2} ~~\text{in}~\Omega,~~u=\sigma~~\text{on }~\partial\Omega
 \end{align*}
 where $q_1,q_2\geq 0, q_1+q_2>1, q_2<2$,  $\sigma\in \mathfrak{M}(\partial\Omega)$ is a Radon measure on $\partial\Omega$. 
\end{abstract}
\section{Introduction and main results}
Let $\Omega$ be  a bounded smooth domain in $\mathbb{R}^N$ or $\Omega=\mathbb{R}_+^{N}:=\mathbb{R}^{N-1}\times(0,\infty)$, $N\geq 3$, and $g:\mathbb{R}\times\mathbb{R}^N\mapsto \mathbb{R}$ be a continuous function. In this paper, we study the solvability of the problem 
\begin{equation}\label{121120149b} \begin{array}{lll}
-\Delta u = g(u,\nabla u)\qquad&\text{in}~\Omega, \\
\phantom{-\Delta}
u = \sigma\qquad&\text{on }~\partial\Omega,\\ 
\end{array}\end{equation}
where $\sigma\in \mathfrak{M}(\partial\Omega)$ is a Radon measure on $\partial\Omega$.
All  solutions are understood in the usual very weak sense, which means that  $u\in L^1_{}(\Omega)$, $g(u,\nabla u)\in L^1_{\rho}(\Omega)$, where $\rho(x)$ is the distance from $x$ to $\partial\Omega$ when $\Omega$ is bounded, or  $u\in L^1_{}(\mathbb{R}_+^{N}\cap B)$, $g(u,\nabla u)\in L^1_{\rho}(\mathbb{R}_+^{N}\cap B)$ for any ball $B$ if $\Omega=\mathbb{R}_+^{N}$, and 
\begin{align}
\int_\Omega u (-\Delta \xi)dx =\int_\Omega g(u,\nabla u)\xi dx-\int_{\partial\Omega}\frac{\partial \xi}{\partial n}d\sigma
\end{align}
for any $\xi \in C^2(\overline{\Omega})\cap C_c(\mathbb{R}^N)$ with $\xi=0$ in $\Omega^c$, where $\rho(x)=\operatorname{dist}(x,\partial\Omega)$, $n$ is the outward unit vector on $\partial\Omega$.
It is well-known that  such a solution  $u$ satisfies 
\begin{align*}
u=\mathbf{G}[g(u,\nabla u)]+\mathbf{P}[\sigma]~~\text{a. e. in } \Omega,
\end{align*}
where $\mathbf{G}[.],\mathbf{P}[.]$,  respectively the Green and the Poisson potentials associated to  $-\Delta$ in $\Omega$, are defined from the Green and the Poisson kernels by 
\begin{align*}
\mathbf{P}[\sigma](y)=\int_{\partial\Omega}\operatorname{P}(y,z)d\sigma(z), ~~\mathbf{G}[g(u,\nabla u)](y)=\int_{\Omega}\operatorname{G}(y,x)g(u,\nabla u)(x)dx,
\end{align*}
see \cite{MV5}.\smallskip 

Our main goal is to establish necessary and sufficient conditions for the existence of weak solutions of \eqref{121120149b} with boundary  measure data, together with sharp pointwise estimates of the solutions. In the sequel we study two cases for the problem 
\eqref{121120149b}:\smallskip

\noindent {\bf 1}- The pure power case

\begin{equation}\label{pow1} \begin{array}{lll}
-\Delta u = |u|^{q-1}u\qquad&\text{in}~\Omega, \\
\phantom{-\Delta}
u = \sigma\qquad&\text{on }~\partial\Omega,\\ 
\end{array}\end{equation}
with $u\geq 0$, $q>1$ and $\sigma\geq 0$.
\smallskip

\noindent {\bf 2}- The mixed gradient-power case
\begin{equation}\label{pow2} \begin{array}{lll}
-\Delta u = |\nabla u|^{q_2}|u|^{q_1-1}u\qquad&\text{in}~\Omega, \\
\phantom{-\Delta}
u = \sigma\qquad&\text{on }~\partial\Omega,\\ 
\end{array}\end{equation}
with $q_1,q_2>0$, $q_1+q_2>1$ and $q_2<2$.\smallskip

The problem \eqref{pow1} has been first studied by Bidaut-V\'eron and Vivier \cite{BiVi} in the subcritical case $1<q<\frac{N+1}{N-1}$ with $\Omega$ bounded. They proved that \eqref{pow1} admits a nonnegative solution provided $\sigma(\partial\Omega)$ is small enough. They also proved that for any $\sigma\in\mathfrak M^+_b(\partial\Omega)$ there holds

\begin{equation}\label{pow3} \begin{array}{lll}
{\bf G}[({\bf P}[\sigma])^q]\leq c\sigma(\partial\Omega){\bf P}[\sigma]
\end{array}\end{equation}
for some $c=c(N,p,q)>0$. Then Bidaut-V\'eron and Yarur \cite{BiYa}  considered again the problem \eqref{pow1} in a bounded domain in a more general situation since they allowed both interior and boundary measure data, giving a complete description of the solutions in the subcritical case, and sufficient conditions for existence in the supercritical case. In particular they showed that
the problem \eqref{pow1} has a solution  if and only if 
\begin{equation}\label{pow4} \begin{array}{lll}
{\bf G}[({\bf P}[\sigma])^q]\leq c{\bf P}[\sigma]
\end{array}\end{equation}
for some $c=c(N,q,\Omega)>0$, see \cite[Th 3.12-3.13, Remark 3.12]{BiYa}. \smallskip

The absorption case, i.e. $g(u,\nabla u)=-|u|^{q-1}u$ has been studied by Gmira and V\'eron \cite{GmV} in the subcritical case (again $1<q<\frac{N+1}{N-1}$) and by Marcus and V\'eron in the supercritical case \cite{MV1}, \cite{MV4}, \cite{MV5}. The case $g(u,\nabla u)=-|\nabla u|^{q}$ was studied by Nguyen Phuoc and V\'eron \cite{NPhV} and extended recently to the case $g(u,\nabla u)=-|\nabla u|^{q_2}|u|^{q_1-1}u$ by Marcus and Nguyen Phuoc \cite{NPhM}. To our knowledge, the problem \eqref{pow2} has not yet been studied.\smallskip

To state our results, let us introduce some notations.  We write $A\lesssim (\gtrsim )B$ if $A\leq  (\geq)C B $ for some $C$ depending on some structural constants,  $ A \asymp B$ if $A\lesssim B\lesssim A$.  Various capacities will be used throughout the paper.  Among them are the Riesz and Bessel capacities  in $\mathbb{R}^{N-1}$  defined respectively by 
  
\begin{align*}
  \operatorname{Cap}_{I_{\gamma},s}(O)=\inf\left\{\int_{\mathbb{R}^{N-1}}f^sdy: f\geq 0, I_\gamma*f\geq \chi_{O} \right\}, 
  \end{align*} 
  \begin{align*}
  \operatorname{Cap}_{G_{\gamma},s}(O)=\inf\left\{\int_{\mathbb{R}^{N-1}}f^sdy: f\geq 0, G_\gamma*f\geq \chi_{O} \right\},
  \end{align*} 
  for any Borel set $O\subset\mathbb{R}^{N-1}$, where $s>1,$ $I_\gamma, G_{\gamma}$ are  the Riesz and  the Bessel kernels in $\mathbb{R}^{N-1}$ with order $\gamma\in (0,N-1)$.  We remark that 
  \begin{align}\label{051220143}
 \operatorname {Cap}_{G_{\gamma},s}(O)\geq\operatorname{Cap}_{I_{\gamma},s}(O)\geq C|O|^{1-\frac{\gamma s}{N-1}}
  \end{align}
  for any Borel set $O\subset \mathbb{R}^{N-1}$ where $\gamma s<N-1$ and $C$ is a positive constant.
  When we consider equations in a bounded smooth domain $\Omega$ in $\mathbb{R}^N$ we use a specific capacity that we define as follows:  there exist open sets $O_1,...,O_m$ in $\mathbb{R}^N$, diffeomorphisms $T_i:O_i\mapsto B_1(0)$ and compact sets $K_1,...,K_m$ in $\partial\Omega$  such that 
  \begin{description}
  \item[a. ] $K_i\subset O_i$,  $\partial\Omega\subset \bigcup\limits_{i = 1}^m K_i$.
  \item[b.] $T_i(O_i\cap \partial\Omega)=B_1(0)\cap \{x_N=0\}$, $T_i(O_i\cap \Omega)=B_1(0)\cap \{x_N>0\}$.
  \item[c.] For any $x\in O_i\cap \Omega$, $\exists y\in O_i\cap\partial\Omega$, $\rho(x)=|x-y|$.
  \end{description} 
  Clearly, $\rho(T_i^{-1}(z))\asymp|z_N|$ for any $z=(z',z_N)\in B_1(0)\cap \{x_N>0\}$ and $|\mathbf{J}_{T_i}(x)|\asymp 1$ for any $x\in  O_i\cap \Omega$, here $\mathbf{J}_{T_i}$ is the Jacobian matrix of $T_i$. \\
  \begin{definition}\label{131120144} Let $\gamma\in (0,N-1), s>1$. We define the $\operatorname{Cap}_{\gamma,s}^{\partial\Omega}$-capacity of a compact set $E\subset \partial\Omega$ by 
  \begin{align*}
  \operatorname{Cap}_{\gamma,s}^{\partial\Omega}(E)=\sum_{i=1}^{m}\operatorname{Cap}_{G_{\gamma},s}(\tilde{T}_i(E\cap K_i)),
  \end{align*}
  where  $T_i(E\cap K_i)=\tilde{T}_i(E\cap K_i)\times\{x_N=0\}$.
  \end{definition}    
  Notice that, if $\gamma s>N-1$ then there exists $C=C(N,\gamma,s,\Omega)>0$ such that \begin{align}\label{051220144}
  \operatorname{Cap}_{\gamma,s}^{\partial\Omega}(\{x\})\geq C
  \end{align} for all $x\in \partial\Omega$. Also the definition does not depend on the choice of the sets $O_i$.\medskip
  
  Our first two theorems give criteria for the solvability of the problem \eqref{121120149b} in $\mathbb{R}^N_+$.
 \begin{theorem}\label{121120142} Let $q>1$ and $\sigma\in \mathfrak{M}_b^+(\mathbb{R}^{N-1})$.  Then,  the following statements are equivalent
   \begin{description}
  \item${\bf 1}$ There exists $C>0$ such that the inequality  
\begin{align}\label{121120143}
\sigma(K)\leq C  \operatorname{Cap}_{I_{\frac{2}{q}},q'}(K)\end{align}
holds for any compact set $K\subset\mathbb{R}^{N-1}$.
   \item${\bf 2}$ There exists $C>0$ such that the relation
 \begin{align}\label{121120144}
\mathbf{G}\left[\left(\mathbf{P}[\sigma]\right)^q\right]\leq C \mathbf{P}[\sigma]<\infty ~~a.e \text{ in }~~ \mathbb{R}^{N}_+
 \end{align}
  holds.\\
    
   \item[3.] The problem 
    \begin{equation}\label{121120145} \begin{array}{lll}
 -\Delta u = u^q ~~&\text{ in }~\mathbb{R}^{N}_+, \\
 \phantom{ -\Delta}
 u = \varepsilon\sigma\quad ~&\text{ in }~\partial\mathbb{R}^{N}_+,\\ 
 \end{array} \end{equation}
   has a positive solution for $\varepsilon>0$ small enough.\medskip\smallskip
    \end{description}  
   \noindent Moreover, there is a constant $C_0>0$ such that if any one
   of the two statement ${\bf 1}$ and ${\bf 2}$  holds with $C\leq C_0$, then equation \eqref{121120145} 
   admits a solution u with $\varepsilon=1$ which satisfies 
   \begin{align}
   u\asymp \mathbf{P}[\sigma].
   \end{align}
   Conversely, if \eqref{121120145} has
   a solution u with $\varepsilon=1$, then the two statements ${\bf 1}$ and ${\bf 2}$ hold for some $C>0$.  
   
 \end{theorem}

As a consequence of Theorem \ref{121120142} when $g(u,\nabla u)=|u|^{q-1}u$ ($q>1$) and $\Omega=\mathbb R^{N}_+$, we prove that  if \eqref{pow1}  has a nonnegative solution $u$ with $\sigma\in \mathfrak M^+_b(\mathbb{R}^{N-1})$, then 
 \begin{align}\label{XX}
\sigma(B_r^{'}(y'))\leq C r^{N-\frac{q+1}{q-1}}
\end{align}
 for any ball $B^{'}_r(y')$ in $\mathbb{R}^{N-1}$ where $C=C(q,N)$ and $q>\frac{N+1}{N-1}$; if $1<q\leq \frac{N+1}{N-1}$, then $\sigma\equiv 0$.  Conversely, if $q>\frac{N+1}{N-1}$, $d\sigma=fdz$ for some $f\geq 0$ which satisfies 
 \begin{align}\label{051220142}
 \int_{B_r^{'}(y')}f^{1+\varepsilon} dz\leq Cr^{N-1-\frac{2(\varepsilon+1)}{q-1}}
 \end{align}
  for some $\varepsilon>0$, then there exists a constant $C_0=C_0(N,q)$ such that \eqref{121120149b}  has a nonnegative solution  if  $C\leq C_0$. The above inequality is an analogue of the classical Fefferman-Phong condition \cite{Fe}.  In particular, \eqref{051220142} holds if $f$ belongs to the Lorentz space $L^{\frac{(N-1)(q-1)}{2},\infty}(\mathbb{R}^{N-1})$.\\
  
  We give sufficient conditions for the existence of weak solutions to  \eqref{121120149b} when $g(u,\nabla u)=|u|^{q_1-1}u|\nabla u|^{q_2}$, $q_1,q_2\geq 0$,  $q_1+q_2>1$ and $q_2<2$.\smallskip
%

   \begin{theorem}\label{1311201437}
    Let $q_1,q_2\geq 0,q_1+q_2>1,q_2<2$ and $\sigma\in \mathfrak M(\mathbb{R}^{N-1})$ such that $\mathbf{P}[|\sigma|]<\infty$ a.e. in $\mathbb{R}^{N-1}$. Assume that 
    there exists $C>0$ such that for any Borel set $K\subset\mathbb{R}^{N-1}$ there holds
    \begin{align}\label{1311201438}
|\sigma|(K)\leq C \operatorname{Cap}_{I_{\frac{2-q_2}{q_1+q_2}},(q_1+q_2)'}(K).\end{align}
Then the problem 
    \begin{equation}\label{1311201439}  \begin{array}{lll}
 -\Delta u = |u|^{q_1-1}u|\nabla u|^{q_2} ~~&\text{ in}~\mathbb{R}^{N}_+, \\
 \phantom{ -\Delta }
 u = \varepsilon\sigma\quad ~&\text{ in }~\partial\mathbb{R}^{N}_+,\\ 
 \end{array}\end{equation}
   has a  solution for $\varepsilon>0$ small enough and it satisfies 
   \begin{align}\label{2110201427}
    |u|\lesssim \mathbf{P}[|\sigma|],~~|\nabla u|\lesssim \rho^{-1} \mathbf{P}[|\sigma|]. \end{align}
\end{theorem}
    
\begin{remark}  In any case and in view of \eqref{051220143},  if $d\sigma=fdz,$ $f\in L^{\frac{(N-1)(q_1+q_2-1)}{2-q_2},\infty}(\mathbb{R}^{N-1})$  and $(N-1)(q_1+q_2-1)>2-q_2$ then \eqref{1311201438} holds for some $C>0$ and the problem \eqref{1311201439} has a solution for $\varepsilon>0$ small enough. However, we can see that condition \eqref{1311201438}  implies $\mathbf{P}[|\sigma|]<\infty$ a.e, see Theorem \ref{2010201420}.
\end{remark}\medskip
    \smallskip
    
In a bounded domain $\Omega$   we obtain existence results  analogous to Theorem \ref{121120142} and \ref{1311201437} provided the capacities on $\partial\Omega$ set in Definition \ref{131120144} are used instead of the Riesz capacities.
 \begin{theorem}\label{121120147} Let $q>1$, $\Omega\subset\mathbb{R}^N$ be a bounded domain with a $C^2$ boundary  and $\sigma\in \mathfrak{M}^+(\partial\Omega)$.  Then,  the following statements are equivalent:
   \begin{description}
   \item${\bf 1}$ There exists $C>0$ such that the inequality  
 \begin{align}\label{121120146}
   \sigma(K)\leq C \operatorname{Cap}^{\partial\Omega}_{\frac{2}{q},q'}(K)\end{align}
for any Borel set $K\subset\partial\Omega$.
\item${\bf 2}$ There exists $C>0$ such that the inequality
  \begin{align}\label{121120148}
 \mathbf{G}\left[\left(\mathbf{P}[\sigma]\right)^q\right]\leq C \mathbf{P}[\sigma]<\infty ~~a.e \text{ in }~~ \Omega,
  \end{align}
   holds.
    \item[3.] The problem 
\begin{equation}\label{121120149} \begin{array}{lll}
  -\Delta u = u^q ~~&\text{ in}~\Omega, 
  \\\phantom{  -\Delta}
  u = \varepsilon\sigma\quad &\text{ on }~\partial\Omega,\\ 
  \end{array} 
  \end{equation}
    admits a positive solution for $\varepsilon>0$ small enough.
      \end{description}  
    Moreover, there is a constant $C_0>0$ such that if any one
    of the two statements ${\bf1}$ and ${\bf 2}$  holds with $C\leq C_0$, then equation \eqref{121120149} 
    has a solution u with $\varepsilon=1$ which satisfies 
    \begin{align}
    u\asymp \mathbf{P}[\sigma].
    \end{align}
    Conversely, if \eqref{121120149} has
    a solution u with $\varepsilon=1$, the above two statements ${\bf 1}$ and ${\bf 2}$ hold for some $C>0$.
  \end{theorem}
  From \eqref{051220144}, we see that if $\sigma\in \mathfrak{M}^+(\partial\Omega)$ and $ 1<q<\frac{N+1}{N-1}$,  then \eqref{121120146} holds for some constant $C>0$. Hence, in this case, the problem \eqref{121120149} has a positive solution for $\varepsilon>0$ small enough.
 \begin{theorem}\label{1311201440}
    Let $q_1,q_2\geq 0,q_1+q_2>1,q_2<2,$ $\Omega\subset \mathbb{R}^N$ be a bounded domain with a $ C^2$ boundary and $\sigma\in \mathfrak{M}(\partial\Omega)$. Assume that there exists $C>0$ such that the inequality
    \begin{align}\label{2010201446}
|\sigma|(K)\leq C \operatorname{Cap}^{\partial\Omega}_{\frac{2-q_2}{q_1+q_2},(q_1+q_2)'}(K)\end{align}
holds for any Borel set $K\subset\partial\Omega$. Then the problem 
 \begin{equation}\label{2010201447} \begin{array}{lll}
 -\Delta u = |u|^{q_1-1}u|\nabla u|^{q_2}\qquad&\text{in}~\Omega, \\
 \phantom{ -\Delta}
 u = \varepsilon\sigma&\text{on }~\partial\Omega,\\ 
 \end{array} \end{equation}
   has a  solution for $\varepsilon>0$ small enough which satisfies \eqref{2110201427}.   
    \end{theorem}
    \begin{remark}  A discussion about the optimality of this condition, as well as the one of Theorem \ref{1311201437}, is conducted in Remark \ref{opt}. We define the subcritical range by
    	\begin{equation}\label{subcri}
    		(N-1)q_1+Nq_2<N+1\quad\text{or equivalently }\;(N-1)(q_1+q_2-1)<2-q_2.
    	\end{equation}
If we assume that we are in the subcritical case, then problem \eqref{2010201447} has a solution for any measure $\sigma\in \mathfrak M_b(\partial\Omega)$ and 
    	$\varepsilon>0$ small enough. 
    \end{remark}
 \section{Integral equations  } Let  $\Omega$ be either $\mathbb{R}^{N-1}\times (0,\infty)$ or $\Omega$  a bounded domain in $\mathbb{R}^N$ with  a $C^2$ boundary $\partial\Omega$.  For $0\leq \alpha\leq \beta <N$, we denote  
 \begin{align}
 \mathbf{N}_{\alpha,\beta}(x,y)=\frac{1}{|x-y|^{N-\beta}\max\left\{|x-y|,\rho(x),\rho(y)\right\}^\alpha}\qquad\forall (x,y)\in \overline{\Omega}\times\overline{\Omega}.   
 \end{align} We set
 \begin{align*}
 \mathbf{N}_{\alpha,\beta}[\nu
 ](x)=\int_{\overline{\Omega}}\mathbf{N}_{\alpha,\beta}(x,y)d\nu(y)\qquad\forall \nu\in\mathfrak{M}^+(\overline{\Omega}),
 \end{align*}
 and denote $\mathbf{N}_{\alpha,\beta}[f]:=\mathbf{N}_{\alpha,\beta}[fdx]$ if $f\in L^{1}_{loc}(\Omega),~f\geq 0$. \smallskip
 
 In this section, we are interested in the solvability of the following integral equations
 \begin{align}
 U=\mathbf{N}_{\alpha,\beta}\left[U^q(\rho(.))^{\alpha_0}\right]+\mathbf{N}_{\alpha,\beta}[\omega]
 \end{align}  where $\alpha_0\geq 0$ and $\omega \in\mathfrak{M}^+(\overline{\Omega})$. \\
 
 We follow the deep ideas developed by Kalton and Verbitsky in \cite{KaVe} who analyzed a PDE problem under the form of an integral equation.
 They proved a certain number of properties of this integral equation which are crucial for our study and, for the sake of completeness, we recall them here.
 Let $X$ be a metric space and $\nu\in \mathfrak{M}^+(X)$. Let $\mathbf{K}$ be a Borel  positive kernel function $\mathbf{K}:X\times X\mapsto (0,\infty]$ such that $\mathbf{K}$ is symmetric and $\mathbf{K}^{-1}$ satisfies a quasi-metric
   inequality, i.e. there is a constant $C\geq 1$ such that for all
   $x,y,z\in X$ we have
   \begin{align*}
   \frac{1}{\mathbf{K}(x,y)}\leq C\left(\frac{1}{\mathbf{K}(x,z)}+\frac{1}{\mathbf{K}(z,y)}\right).
   \end{align*}
   Under these conditions, we can define  the quasi-metric $d$ by
   $$d(x,y)=\frac{1}{\mathbf{K}(x,y)},$$
and denote by    $\mathbb{B}_{r}(x)=\{y\in X\!:d(x,y)<r\}$ the open $d$-ball of radius $r>0$ and center $x$. Note that this set  can be empty. \smallskip

   For $\omega\in\mathfrak{M}^+(X)$, we define the potentials $\mathbf{K}\omega$ and $\mathbf{K}^{\nu}f$ by 
  
   $$\mathbf{K}\omega(x)=\int_{X}\mathbf{K}(x,y)d\omega(y),~~\mathbf{K}^{\nu}f(x)=\int_{X}\mathbf{K}(x,y)f(y)d\nu(y),$$    
  and for $q>1$,  the capacity 
$\operatorname{Cap}^\nu_{\mathbf{K},q'}$ in $X$  by 
  \begin{align*}
  \operatorname{Cap}^\nu_{\mathbf{K},q'}(E)=\inf\left\{\int_{X}g^{q'}d\nu: g\geq 0 , \mathbf{K}^\nu g\geq \chi_E\right\},
  \end{align*} 
   for any Borel set $E\subset X$. 
  \begin{theorem}[\cite{KaVe}] \label{151020148} Let $q>1$ and $\nu, \omega\in\mathfrak{M}^+(X)$ such that
    \begin{align}\label{1510201410}
 &~~~~~~~~~ \int_{0}^{2r}\frac{\nu(\mathbb{B}_s(x))}{s}\frac{ds}{s}\leq C \int_{0}^{r}\frac{\nu(\mathbb{B}_s(x))}{s}\frac{ds}{s},\\[4mm]
 &
 \sup_{y\in \mathbb{B}_r(x)}\int_{0}^{r}\frac{\nu(\mathbb{B}_s(y))}{s}\frac{ds}{s}\leq C \int_{0}^{r}\frac{\nu(\mathbb{B}_s(x))}{s}\frac{ds}{s},\label{1510201410*}
  \end{align}
   for any $r>0,x\in X$, where $C>0$ is a constant.   Then the following statements are equivalent:
  \begin{description}
  \item${\bf 1}$ The equation $u= \mathbf{K}^\nu u^q+\varepsilon \mathbf{K}\omega$ has a solution  for some $\varepsilon>0$.
  
  \item${\bf 2}$  The inequality 
  \begin{align}\label{151020143}
   \int_{E}(\mathbf{K}\omega_E)^{q}d\sigma\leq C\omega(E)
   \end{align}
  holds   for any Borel set $E\subset X$ where $\omega_E=\chi_E\omega$. 
   \item[3.] For any Borel set $E\subset X$, there holds
   \begin{align}\label{1510201414}
\omega(E)\leq C \operatorname{Cap}^\nu_{\mathbf{K},q'}(E).
\end{align}
  \item[4.] The inequality 
  \begin{align}\label{151020144}
  \mathbf{K}^\nu\left(\mathbf{K}\omega\right)^q\leq C \mathbf{K}\omega<\infty ~~\nu-a.e.
  \end{align}
  holds.
  \end{description}
   \end{theorem}
We check below that $N_{\alpha,\beta}$ satisfies all assumptions of $\mathbf{K}$ in Theorem \ref{151020148}.                                        
    \begin{lemma}\label{201020141}
$\mathbf{N}_{\alpha,\beta}$ is symmetric and  satisfies the quasi-metric inequality.
   \end{lemma}
\begin{proof} Clearly, $\mathbf{N}_{\alpha,\beta}$ is symmetric. Now we check the quasi-metric inequality associated to $\mathbf{N}_{\alpha,\beta}$ and $X=\overline{\Omega}$.  For any $x,z,y\in \overline{\Omega}$ such that $x\not= y\not= z$, we have
\begin{align*}
|x-y|^{N-\beta+\alpha}&\lesssim |x-z|^{N-\beta+\alpha}+|z-y|^{N-\beta+\alpha}\\&\lesssim \frac{1}{\mathbf{N}_{\alpha,\beta}(x,z)}+\frac{1}{\mathbf{N}_{\alpha,\beta}(z,y)}.
\end{align*}    
   Since $|\rho(x)-\rho(y)|\leq |x-y|$, there holds 
\begin{align*}
   & |x-y|^{N-\beta}(\rho(x))^\alpha+|x-y|^{N-\beta}(\rho(y))^\alpha\lesssim |x-y|^{N-\beta}(\min\{\rho(x),\rho(y)\})^\alpha+|x-y|^{N-\beta+\alpha}\\&~~~~~~~\lesssim \left(|x-z|^{N-\beta}+|z-y|^{N-\beta}\right)(\min\{\rho(x),\rho(y)\})^\alpha+|x-z|^{N-\beta+\alpha}+|z-y|^{N-\beta+\alpha}\\&~~~~~~\lesssim\left((\rho(x))^\alpha|x-z|^{N-\beta}+|x-z|^{N-\beta+\alpha}\right)+\left((\rho(y))^\alpha|z-y|^{N-\beta}+|z-y|^{N-\beta+\alpha}\right)\\&~~~~~~~\lesssim \frac{1}{\mathbf{N}_{\alpha,\beta}(x,z)}+\frac{1}{\mathbf{N}_{\alpha,\beta}(z,y)}.
   \end{align*}Thus,   
  \begin{align*}
 \frac{1}{\mathbf{N}_{\alpha,\beta}(x,y)} \lesssim \frac{1}{\mathbf{N}_{\alpha,\beta}(x,z)}+\frac{1}{\mathbf{N}_{\alpha,\beta}(z,y)}.
  \end{align*}   
   \end{proof}\medskip
   
   Next we give sufficient conditions for \eqref{1510201410}, \eqref{1510201410*} to hold, in view of the applications that we develop in Sections 3 and 4.
    \begin{lemma}\label{011220141} If $d\nu(x)= (\rho(x))^{\alpha_0}\chi_\Omega dx$ with $\alpha_0\geq 0$, then \eqref{1510201410} and \eqref{1510201410*} hold. 
    
    \end{lemma}
\begin{proof} It is easy to see that for any $x\in \overline{\Omega},~s>0$
\begin{align}
B_{2^{-\frac{\alpha+1}{N-\beta}}S}(x)\cap\overline{\Omega}\subset \mathbb{B}_s(x)\subset B_{S}(x)\cap\overline{\Omega},
\end{align}
    with  $S=\min\{s^{\frac{1}{N-\beta+\alpha}},s^{\frac{1}{N-\beta}}(\rho(x))^{-\frac{\alpha}{N-\beta}}\}$ and $\mathbb{B}_s(x)=\overline{\Omega}$ when $s>2^{\frac{\alpha N}{N-\alpha}}(diam\,(\Omega))^N$.\\ 
We show that    for any $0\leq s<8diam\,(\Omega)$, $x\in\overline{\Omega}$ 
  \begin{align}\label{161020147}
  \nu(B_s(x))\asymp (\max\{\rho(x),s\})^{\alpha_0} s^N.
  \end{align}Indeed, take $0\leq s<8diam\,(\Omega)$, $x\in\overline{\Omega}$. 
  There exist  $\varepsilon=\varepsilon(\Omega)\in (0,1)$ and $x_s\in \Omega$ such that  $B_{\varepsilon s}(x_s)\subset B_s(x)\cap \Omega$ and $\rho(x_s)>\varepsilon s$. \smallskip

\noindent  (a) If $0\leq s\leq \frac{\rho(x)}{4}$, so for any $y\in B_s(x)$, $
 \rho(y)\asymp \rho(x). 
$
Thus we obtain \eqref{161020147} because 
$$\nu(B_s(x))\asymp (\rho(x))^{\alpha_0} |B_s(x)\cap\Omega|\asymp (\rho(x))^{\alpha_0} s^N.$$ 

\noindent (b)  If $s>\frac{\rho(x)}{4}$, since $\rho(y)\leq \rho(x)+|x-y|<5s$ for any $y\in B_s(x)$, there holds $\nu(B_s(x))\lesssim s^{N+\alpha_0}$ and we have the following dichotomy: \smallskip

\noindent(b.1) either $s\leq 4\rho(x) $, then 
   $$\nu(B_s(x))\gtrsim \nu(B_{\frac{\rho(x)}{4}}(x))\asymp (\rho(x))^{\alpha_0+N} \gtrsim s^{N+\alpha_0} ;$$
   
    (b.2)   or $s\geq 4\rho(x)$, 
 we have for any $y\in B_{\varepsilon s/2}(x_s)$, $\rho(y)\geq -|y-x_s|+\rho(x_s)>\varepsilon s/2 $. It follows 
\begin{align*}
\nu(B_s(x))\gtrsim \nu(B_{\varepsilon s/2}(x_s))\gtrsim  s^{N+\alpha_0}.
\end{align*}
Therefore \eqref{161020147} holds.\smallskip

\noindent Next,  for any $0\leq s<2^{\frac{(\alpha+1)(N-\beta+\alpha)}{N-\beta}}(diam\,(\Omega))^{N-\beta+\alpha}$ and $x\in\overline{\Omega}$,  we have 
 \begin{align*}
    \nu(\mathbb{B}_s(x))&\asymp (\max\{\rho(x),\min\{s^{\frac{1}{N-\beta+\alpha}},s^{\frac{1}{N-\beta}}(\rho(x))^{-\frac{\alpha}{N-\beta}}\}\})^{\alpha_0} \\&~~\times\left(\min\{s^{\frac{1}{N-\beta+\alpha}},s^{\frac{1}{N-\beta}}(\rho(x))^{-\frac{\alpha}{N-\beta}}\}\right)^N\\&\asymp \left\{ \begin{array}{l}
s^{\frac{\alpha_0+N}{N-\beta+\alpha}}  ~~~~~~~~~~~~~~~~~\text{if}~\rho(x)\leq s^{\frac{1}{N-\beta+\alpha}}, \\
(\rho(x))^{\alpha_0-\frac{\alpha N}{N-\beta}}s^{\frac{N}{N-\beta}}~ ~\text{if }~\rho(x)\geq  s^{\frac{1}{N-\beta+\alpha}},\\ 
  \end{array} \right.
\end{align*}
and $ \nu(\mathbb{B}_s(x))=\nu(\overline{\Omega})\asymp (diam\,(\Omega))^{\alpha_0+N}$ if $s>2^{\frac{(\alpha+1)(N-\beta+\alpha)}{N-\beta}}(diam\,(\Omega))^{N-\beta+\alpha}$.
We get, 
\begin{align*}
\int_{0}^{r}\frac{\nu(\mathbb{B}_s(x))}{s}\frac{ds}{s}\asymp \left\{ \begin{array}{l}
(diam\,(\Omega))^{\alpha_0+\beta-\alpha}  ~~~\text{if}~r> (diam\,(\Omega))^{N-\beta+\alpha}, \\ 
  r^{\frac{\alpha_0+\beta-\alpha}{N-\beta+\alpha}}  ~~~~~~~~~~~~~~~~\text{if}~r\in ((\rho(x))^{N-\beta+\alpha},(diam\,(\Omega))^{N-\beta+\alpha}], \\
(\rho(x))^{\alpha_0-\frac{\alpha N}{N-\beta}}r^{\frac{\beta}{N-\beta}}~ ~\text{if }~r\in (0,(\rho(x))^{N-\beta+\alpha}].\\ 
    \end{array} \right.
\end{align*}
Therefore \eqref{1510201410} holds. It remains to prove  \eqref{1510201410*}. For any $x\in \overline{\Omega}$ and $r>0$,  it is clear that if $r>\frac{1}{2}(\rho(x))^{N-\beta+\alpha}$ we have  
   \begin{align*}
\sup_{y\in \mathbb{B}_r(x)}\int_{0}^{r}\frac{\nu(\mathbb{B}_s(y))}{s}\frac{ds}{s}\lesssim \min\{r^{\frac{\alpha_0+\beta-\alpha}{N-\beta+\alpha}},(diam\,(\Omega))^{\alpha_0+\beta-\alpha}  \},
\end{align*} 
from which inequality   we obtain 
\begin{align*}
 \sup_{y\in \mathbb{B}_r(x)}\int_{0}^{r}\frac{\nu(\mathbb{B}_s(y))}{s}\frac{ds}{s}\lesssim \int_{0}^{r}\frac{\nu(\mathbb{B}_s(x))}{s}\frac{ds}{s}.
 \end{align*} 
   If $0<r\leq \frac{1}{2}(\rho(x))^{N-\beta+\alpha}$, we have $\mathbb{B}_r(x)\subset B_{r^{\frac{1}{N-\beta}}(\rho(x))^{-\frac{\alpha}{N-\beta}}}(x)$  and $\rho(x)\asymp\rho(y)$ for all $y\in   B_{r^{\frac{1}{N-\beta}}(\rho(x))^{-\frac{\alpha}{N-\beta}}}(x)$,  thus
\begin{align*}
  \sup_{y\in \mathbb{B}_r(x)}\int_{0}^{r}\frac{\nu(\mathbb{B}_s(y))}{s}\frac{ds}{s}&\leq  \sup_{|y-x|<r^{\frac{1}{N-\beta}}(\rho(x))^{-\frac{\alpha}{N-\beta}}}\int_{0}^{r}\frac{\nu(\mathbb{B}_s(y))}{s}\frac{ds}{s}\\& \asymp  \sup_{|y-x|<r^{\frac{1}{N-\beta}}(\rho(x))^{-\frac{\alpha}{N-\beta}}}(\rho(y))^{\alpha_0-\frac{\alpha N}{N-\beta}}r^{\frac{\beta}{N-\beta}} 
  \\& \asymp  (\rho(x))^{\alpha_0-\frac{\alpha N}{N-\beta}}r^{\frac{\beta}{N-\beta}}   \\& \asymp  \int_{0}^{r}\frac{\nu(\mathbb{B}_s(x))}{s}\frac{ds}{s}.  
  \end{align*}
  Therefore, \eqref{1510201410*} holds.   
    \end{proof}
 \begin{remark} Lemma \ref{201020141} and \ref{011220141} in the case $\alpha=\beta=2$ and $\alpha_0=q+1$ had already been proved by Kalton and Verbitsky in \cite{KaVe}. 
 \end{remark}
    \begin{definition}\label{zz}
    For $\alpha_0\geq 0, 0\leq \alpha\leq \beta<N$ and $ s>1$, we define $\operatorname{Cap}^{\alpha_0}_{\mathbf{N}_{\alpha,\beta},s}$ by
\begin{align*}
 \operatorname{Cap}^{\alpha_0}_{\mathbf{N}_{\alpha,\beta},s}(E)=\inf\left\{\int_{\overline{\Omega}}g^{s}(\rho(x))^{\alpha_0}dx: g\geq 0 , \mathbf{N}_{\alpha,\beta}[g(\rho(.))^{\alpha_0}]\geq \chi_E\right\}
 \end{align*}
 for any Borel set $E\subset\overline{\Omega}$.
    \end{definition} 
    Clearly,  we have
\begin{align*}
\operatorname{Cap}^{\alpha_0}_{\mathbf{N}_{\alpha,\beta},s}(E)=\inf\left\{\int_{\overline{\Omega}}g^{s}(\rho(x))^{-\alpha_0(s-1)}dx: g\geq 0 , \mathbf{N}_{\alpha,\beta}[g]\geq \chi_E\right\}
   \end{align*}
   for any Borel set $E\subset\overline{\Omega}$. Furthermore we have by \cite[Theorem 2.5.1]{55AH},  
\begin{align}\label{051220141}
\left(\operatorname{Cap}^{\alpha_0}_{\mathbf{N}_{\alpha,\beta},s}(E)\right) ^{1/s}=\sup\left\{\omega (E):\omega\in \mathfrak{M}_b^+(E),||\mathbf{N}_{\alpha,\beta}[\omega]||_{L^{s'}(\Omega,(\rho(.)))^{\alpha_0}dx)}\leq 1\right\}
\end{align}
for any compact set $E\subset\overline{\Omega}$, where $s'$ is the conjugate exponent of $s$.\medskip

Thanks to Lemma \ref{201020141} and \ref{011220141} , we can apply Theorem \ref{151020148} and we obtain:  
 \begin{theorem}\label{2010201420}
 Let $\omega\in\mathfrak{M}^+(\overline{\Omega})$,  $\alpha_0\geq 0, 0\leq \alpha\leq \beta<N$ and  $q>1$.  Then the following statements are equivalent:
   \begin{description}
   \item${\bf 1}$ The equation $u= \mathbf{N}_{\alpha,\beta}[ u^q(\rho(.))^{\alpha_0}]+\varepsilon \mathbf{N}_{\alpha,\beta}[ \omega]$ has a solution  for $\varepsilon>0$ small enough.   
   \item${\bf 2}$  The inequality 
   \begin{align}\label{201020143}
    \int_{E\cap\Omega}(\mathbf{N}_{\alpha,\beta}[\omega_E])^{q}(\rho(x))^{\alpha_0}dx\leq C\omega(E)
    \end{align}
   holds  for some $C>0$ and any Borel set $E\subset\overline{\Omega}$, $\omega_E=\omega\chi_E$.
 \item[3.] The inequality  \begin{align}\label{201020147}
   \omega(K)\leq C \operatorname{Cap}^{\alpha_0}_{\mathbf{N}_{\alpha,\beta},q'}(K)
   \end{align}
   holds for some $C>0$ and any compact set $K\subset\overline{\Omega}$.    
   \item[4.] The inequality 
   \begin{align}\label{201020146}
   \mathbf{N}_{\alpha,\beta}\left[\left(\mathbf{N}_{\alpha,\beta}[\omega]\right)^q(\rho(.))^{\alpha_0}\right]\leq C\mathbf{N}_{\alpha,\beta}[\omega]<\infty ~~a.e \text{ in }~~ \Omega
   \end{align}
   holds for some $C>0$.
   \end{description}
 \end{theorem}    
   To apply the previous theorem we need the following result.    
   \begin{proposition}\label{2010201434}Let $q>1$, $\nu, \omega\in \mathfrak{M}^+(X)$. Suppose that $A_1,A_2,B_1,B_2:X\times X\mapsto [0,+\infty)$ are Borel positive Kernel functions with $A_1\asymp A_2,B_1\asymp B_2$. Then, the following statements are equivalent:   \begin{description}
   \item${\bf 1}$ The equation $u=A^\nu_1u^q+\varepsilon B_1\omega\;$ $\nu$-a.e  has a solution for $\varepsilon>0$ small enough.
   \item${\bf 2}$ The equation $u=A^\nu_2u^q+\varepsilon B_2\omega\;$ $\nu-$a.e has a solution for $\varepsilon>0$ small enough.
   \item[3.] The problem $u\asymp A^\nu_1u^q+\varepsilon B_1\omega\;$ $\nu$-a.e has a solution  for $\varepsilon>0$ small enough.
   \item[4.] The equation $u\gtrsim A^\nu_1u^q+\varepsilon B_1\omega\;$ $\nu$-a.e has a solution  for $\varepsilon>0$ small enough.
      \end{description}
   \end{proposition}
\begin{proof} We prove only that 4 implies 2.  Suppose that there exist $c_1>0,\varepsilon_0>0$ and a position Borel function $u$ such that 
   \begin{align*}
   A^\nu_1u^q+\varepsilon_0 B_1\omega\leq c_1 u. 
   \end{align*}
   Taken $c_2>0$ with  $A_2\leq c_2A_1,B_2\leq c_2B$. We consider $u_{n+1}=A^\nu_2u_n^q+\varepsilon_0(c_1c_2)^{-\frac{q}{q-1}} B_2\omega$ and $u_0=0$ for any $n\geq 0$. Clearly, $u_{n}\leq (c_1c_2)^{-\frac{1}{q-1}}u$ for any $n$ and $\{u_n\}$ is nondecreasing. Thus, $U=\lim\limits_{n\to \infty }u_n$ is a solution of $U=A^\nu_2U^q+\varepsilon_0(c_1c_2)^{-\frac{q}{q-1}} B_2\omega$.
   \end{proof}\medskip
   
   The following results provide some relations between the capacities  $\operatorname{Cap}^{\alpha_0}_{\mathbf{N}_{\alpha,\beta},s}$ and the Riesz capacities on $\mathbb{R}^{N-1}$ which allow to define the capacities on $\partial\Omega$. 

    \begin{proposition}\label{2010201417}
Assume that $\Omega=\mathbb{R}^{N-1}\times (0,\infty)$ and let $\alpha_0\geq 0$ such that 
$$-1+s'(1+\alpha-\beta)<\alpha_0<-1+s'(N-\beta+\alpha).$$ 
There holds 
\begin{align}\label{201020149}
\operatorname{Cap}^{\alpha_0}_{\mathbf{N}_{\alpha,\beta},s}(K\times\{0\})\asymp \operatorname{Cap}_{I_{\beta-\alpha+\frac{\alpha_0+1}{s'}-1},s'}(K)
\end{align}
 for any compact set $K\subset\mathbb{R}^{N-1},$
    \end{proposition}
\begin{proof} The proof relies on an idea of  \cite[Corollary 4.20]{QH}. Thanks to \cite[Theorem 2.5.1]{55AH} and \eqref{051220141},  we get \eqref{201020149} from the following estimate:
for any $\mu\in \mathfrak{M}_b^+(\mathbb{R}^{N-1})$
   \begin{align}\label{2010201410} ||\mathbf{N}_{\alpha,\beta}[\mu\otimes\delta_{\{x_N=0\}}]||_{L^{s'}(\Omega,(\rho(.)))^{\alpha_0}dx)}\asymp|| I_{\beta-\alpha+\frac{\alpha_0+1}{s'}-1}[\mu]||_{L^{s'}(\mathbb{R}^{N-1})},
\end{align}    
where  $I_{\gamma}[\mu]$ is the Riesz potential of $\mu$ in $\mathbb{R}^{N-1}$, i.e
  \begin{align*}I_{\gamma}[\mu](y)=\int_{0}^{\infty}\frac{\mu(B'_r(y))}{r^{N-1-\gamma}}\frac{dr}{r}~~\forall~y\in\mathbb{R}^{N-1},\end{align*} with $B'_r(y)$ being a ball in $\mathbb{R}^{N-1}$. We have 
   \begin{align*}
   ||\mathbf{N}_{\alpha,\beta}[\mu\otimes\delta_{\{x_N=0\}}]||_{L^{s'}(\Omega,(\rho(.))^{\alpha_0}dx)}^{s'}&=\int_{\mathbb{R}^{N-1}}\int_{0}^{\infty}\left(\int_{\mathbb{R}^{N-1}}\frac{d\mu(z)}{(|x'-z|^2+x_N^2)^{\frac{N-\beta+\alpha}{2}}}\right)^{s'}x_N^{\alpha_0}dx_Ndx'\\&
 \asymp\int_{\mathbb{R}^{N-1}}\int_{0}^{\infty}\left(\int_{x_N}^{\infty}\frac{\mu(B'_r(x'))}{r^{N-\beta+\alpha}}\frac{dr}{r}\right)^{s'}x_N^{\alpha_0}dx_Ndx'.
   \end{align*}
   Notice that 
   \begin{align*}
   \int_{0}^{\infty}\left(\int_{x_N}^{\infty}\frac{\mu(B'_r(x'))}{r^{N-\beta+\alpha}}\frac{dr}{r}\right)^{s'}x_N^{\alpha_0}dx_N&\geq  \int_{0}^{\infty}\left(\int_{x_N}^{2x_N}\frac{\mu(B'_r(x'))}{r^{N-\beta+\alpha}}\frac{dr}{r}\right)^{s'}x_N^{\alpha_0}dx_N\\& \gtrsim
   \int_{0}^{\infty}\left(\frac{\mu(B'_{x_N}(x'))}{x_N^{N-\beta+\alpha-\frac{\alpha_0+1}{s'}}}\right)^{s'}\frac{d x_N}{x_N}.
   \end{align*}
On the other hand, using H\"older's inequality and Fubini's Theorem, we obtain
  \begin{align*}
  &\int_{0}^{\infty}\left(\int_{x_N}^{\infty}\frac{\mu(B'_r(x'))}{r^{N-\beta+\alpha}}\frac{dr}{r}\right)^{s'}x_N^{\alpha_0}dx_N\leq \int_{0}^{\infty}\left(\int_{x_N}^{\infty}r^{-\frac{s}{2s'}}\frac{dr}{r}\right)^{\frac{s'}{s}}\int_{x_N}^{\infty}\left(\frac{\mu(B'_r(x'))}{r^{N-\beta+\alpha-\frac{1}{2s'}}}\right)^{s'}\frac{dr}{r}x_N^{\alpha_0}dx_N\\[2mm]&~~~~~~~~~~~~~~~~~~~~~~~~~~~~~~~~~~~~~~~~~~~~~~~
  =C
  \int_{0}^{\infty}\int_{x_N}^{\infty}\left(\frac{\mu(B'_r(x'))}{r^{N-\beta+\alpha-\frac{1}{2s'}}}\right)^{s'}\frac{dr}{r}x_N^{\alpha_0-\frac{1}{2}}dx_N \\[2mm]&~~~~~~~~~~~~~~~~~~~~~~~~~~~~~~~~~~~~~~~~~~~~~~~
  =C \int_{0}^{\infty}\int_{0}^{r}x_N^{\alpha_0-\frac{1}{2}}dx_N\left(\frac{\mu(B'_r(x'))}{r^{N-\beta+\alpha-\frac{1}{2s'}}}\right)^{s'}\frac{dr}{r}\\[2mm]&~~~~~~~~~~~~~~~~~~~~~~~~~~~~~~~~~~~~~~~~~~~~~~~
  =C \int_{0}^{\infty}\left(\frac{\mu(B'_{r}(x'))}{r^{N-\beta+\alpha-\frac{\alpha_0+1}{s'}}}\right)^{s'}\frac{d r}{r}.
  \end{align*} 
  Thus,   
    \begin{align}
  ||\mathbf{N}_{\alpha,\beta}[\mu\otimes\delta_{\{x_N=0\}}]||_{L^{s'}(\Omega,(\rho(.)))^{\alpha_0}dx)}\asymp\left(\int_{\mathbb{R}^{N-1}}\int_{0}^{\infty}\left(\frac{\mu(B'_{r}(y))}{r^{N-\beta+\alpha-\frac{\alpha_0+1}{s'}}}\right)^{s'}\frac{d r}{r}dy\right)^{1/s'}.
  \end{align}
 It implies  \eqref{2010201410} from \cite[Theorem 2.3]{55VHV}. 
    \end{proof}  \medskip\\  

    \begin{proposition}\label{2010201414} Let $\Omega\subset \mathbb{R}^N$ be a bounded domain a  $ C^2$ boundary. Assume  $\alpha_0\geq 0$ and  $-1+s'(1+\alpha-\beta)<\alpha_0<-1+s'(N-\beta+\alpha)$. Then there holds
\begin{align}\label{121120141}
    \operatorname{Cap}^{\alpha_0}_{\mathbf{N}_{\alpha,\beta},s}(E)\asymp \operatorname{Cap}^{\partial\Omega}_{\beta-\alpha+\frac{\alpha_0+1}{s'}-1,s}(E)
    \end{align}
 for any compact set $E\subset\partial\Omega\subset\mathbb{R}^N.$
    \end{proposition}
\begin{proof} Let $K_1,...,K_m$ be as in definition \ref{131120144}. We have 
  \begin{align*}
 \operatorname{Cap}^{\alpha_0}_{\mathbf{N}_{\alpha,\beta},s}(E)\asymp \sum_{i=1}^{m}\operatorname{Cap}^{\alpha_0}_{\mathbf{N}_{\alpha,\beta},s}(E\cap K_i),
   \end{align*}
   for any  compact set $E\subset\partial\Omega.$ By definition \ref{131120144}, we need to prove that 
    \begin{align}
    \operatorname{Cap}^{\alpha_0}_{\mathbf{N}_{\alpha,\beta},s}(E\cap K_i)\asymp \operatorname{Cap}_{G_{\beta-\alpha+\frac{\alpha_0+1}{s'}-1},s}(\tilde{T}_i(E\cap K_i))~~\forall~i=1,2,...,m.
    \end{align}    
   We can show that  for any $\omega\in\mathfrak{M}_b^+(\partial\Omega)$ and $i=1,...,m$, there exists $\omega_i\in \mathfrak{M}_b^+(\tilde{T}_i(K_i))$ with $T_i(K_i)=\tilde{T}_i(K_i)\times\{x_N=0\}$ such that 
\begin{align*}
\omega_{i}(O)=\omega(T_i^{-1}(O\times \{0\}))
\end{align*}
for all Borel set $O\subset \tilde{T}_i(K_i)$, its proof can be found in \cite[Proof of Lemma 5.2.2]{55AH}.  Thanks to \cite[Theorem 2.5.1]{55AH}, it is enough to show that for any $i\in \{1,2,...,m\}$ there holds
  \begin{align}\label{2010201412}
  ||\mathbf{N}_{\alpha,\beta}[\chi_{K_i}\omega]||_{L^{s'}(\Omega,(\rho(.)))^{\alpha_0}dx)}\asymp|| G_{\beta-\alpha+\frac{\alpha_0+1}{s'}-1}[\omega_i]||_{L^{s'}(\mathbb{R}^{N-1})},
  \end{align}  
   where  $G_{\gamma}[\omega_i]~(0<\gamma<N-1)$ is the Bessel potential of $\omega_i$ in $\mathbb{R}^{N-1}$, i.e
   \begin{align*}
   G_{\gamma}[\omega_i](x)=\int_{\mathbb{R}^{N-1}}G_{\gamma}(x-y)d\omega_i(y). 
   \end{align*} 
  Indeed, we have 
  \begin{align*}
 & ||\mathbf{N}_{\alpha,\beta}[\omega\chi_{K_i}]||_{L^{s'}(\Omega,(\rho(.)))^{\alpha_0}dx)}^{s'}=\int_{\Omega}\left(\int_{K_i}\frac{d\omega(z)}{|x-z|^{N-\beta+\alpha}}\right)^{s'}(\rho(x))^{\alpha_0}dx
  \\&~~~~~=\int_{O_i\cap \Omega}\left(\int_{K_i}\frac{d\omega(z)}{|x-z|^{N-\beta+\alpha}}\right)^{s'}(\rho(x))^{\alpha_0}dx+\int_{\Omega\backslash O_i}\left(\int_{K_i}\frac{d\omega(z)}{|x-z|^{N-\beta+\alpha}}\right)^{s'}(\rho(x))^{\alpha_0}dx
  \\&~~~~~\asymp\int_{O_i\cap \Omega}\left(\int_{K_i}\frac{d\omega(z)}{|x-z|^{N-\beta+\alpha}}\right)^{s'}(\rho(x))^{\alpha_0}dx+(\omega(K_i))^{s'}.
  \end{align*}
  Here we used $|x-z|\asymp 1$ for any $x\in \Omega\backslash O_i, z\in K_i$. \\
By  a standard change of variable we obtain
  \begin{align*}
 & \int_{O_i\cap \Omega}\left(\int_{K_i}\frac{d\omega(z)}{|x-z|^{N-\beta+\alpha}}\right)^{s'}(\rho(x))^{\alpha_0}dx+(\omega(K_i))^{s'}\\&~~~~~=\int_{T_i(O_i\cap \Omega)}\left(\int_{K_i}\frac{d\omega(z)}{|T_i^{-1}(y)-z|^{N-\beta+\alpha}}\right)^{s'}(\rho(T_i^{-1}(y)))^{\alpha_0}|\mathbf{J}_{T_i}(T_i^{-1}(y))|^{-1}dy+(\omega(K_i))^{s'}
   \\&~~~~~\asymp\int_{B_1(0)\cap \{x_N>0\}}\left(\int_{K_i}\frac{d\omega(z)}{|y-T_i(z)|^{N-\beta+\alpha}}\right)^{s'}y_N^{\alpha_0}dy+(\omega(K_i))^{s'}~\text{ with } y=(y',y_N),
  \end{align*}
  since $|T_i^{-1}(y)-z|\asymp |y-T_i(z)| $, $|\mathbf{J}_{T_i}(T_i^{-1}(y))|\asymp 1$ and $\rho(T_i^{-1}(y))\asymp y_N$ for all $(y,z)\in T_i(O_i\cap \Omega)\times  K_i$. \
 From the definition of $\omega_i$, we have 
  \begin{align*}
  &\int_{B_1(0)\cap \{x_N>0\}}\left(\int_{K_i}\frac{1}{|y-T_i(z)|^{N-\beta+\alpha}}d\omega(z)\right)^{s'}y_n^{\alpha_0}dy+(\omega(K_i))^{s'}
   \\&~~~=\int_{B_1(0)\cap \{x_N>0\}}\left(\int_{\tilde{T}_i(K_i)}\frac{1}{(|y'-\xi|^2+y_N^2)^{\frac{{N-\beta+\alpha}}{2}}}d\omega_i(\xi)\right)^{s'}y_N^{\alpha_0}dy_Ndy'+(\omega(K_i))^{s'}\\&~~~\asymp \int_{\mathbb{R}^{N-1}}\int_{0}^{\infty}\left(\int_{\min\{y_N,R\}}^{2R}\frac{\omega_i(B'_r(y'))}{r^{N-\beta+\alpha}}\frac{dr}{r}\right)^{s'}y_N^{\alpha_0}dy_Ndy'~~\text{ with }~R=\operatorname{diam\,}(\Omega).
  \end{align*}
  As in the proof of Proposition \ref{2010201417}, there holds 
  \begin{align*}
  &\int_{\mathbb{R}^{N-1}}\int_{0}^{\infty}\left(\int_{\min\{y_N,R\}}^{2R}\frac{\omega_i(B'_r(y'))}{r^{N-\beta+\alpha}}\frac{dr}{r}\right)^{s'}y_N^{\alpha_0}dy_Ndy'
 \\& ~~~~~\asymp \int_{\mathbb{R}^{N-1}}\int_{0}^{2R}\left(\frac{\omega_i(B'_r(y'))}{r^{{N-\beta+\alpha}-\frac{\alpha_0+1}{s'}}}\right)^{s'}\frac{d r}{r}dy'. \end{align*}
 Therefore, we get \eqref{2010201412} from \cite[Theorem 2.3]{55VHV}. 
  \end{proof}
  \begin{remark} Proposition \ref{2010201417} and \ref{2010201414} with $\alpha=\beta=2,\alpha_0=q+1$ were demonstrated by Verbitsky in \cite[Apppendix B]{Dyn}, using an alternative approach. 
  \end{remark}
 \section{Proof of the main results}
 We denote 
 \begin{align*}
 \mathbf{P}[\sigma](x)=\int_{\partial\Omega}\operatorname{P}(x,z)d\sigma(z), ~~\mathbf{G}[f](x)=\int_{\Omega}\operatorname{G}(x,y)f(y)dy
 \end{align*}
 for any $\sigma\in \mathfrak{M}(\partial\Omega), f\in L^1_{\rho}(\Omega), f\geq 0$. Then  the unique  weak solution of 
 $$\begin{array}{lll}
 -\Delta u=f\qquad&\text {in }\Omega,\\
 \phantom{ -\Delta}
 u=\sigma\qquad&\text {on }\partial\Omega,
 \end{array}$$ 
 can be represented by   
 \begin{align*}
 u(x)=\mathbf{G}[f](x)+\mathbf{P}[\sigma](x)~~\forall ~x\in\Omega.
 \end{align*}
 We recall below some classical estimates for the Green and the Poisson kernels.
  \begin{align*}
  & \operatorname{G}(x,y)\asymp \min\left\{\frac{1}{|x-y|^{N-2}}, \frac{\rho(x)\rho(y)}{|x-y|^{N}}\right\},\\&
  \operatorname{P}(x,z)\asymp \frac{\rho(x)}{|x-z|^{N}},
  \end{align*}
  and 
  \begin{align*}
  |\nabla_x \operatorname{G}(x,y)|\lesssim \frac{\rho(y)}{|x-y|^{N}}\min\left\{1,\frac{|x-y|}{\sqrt{\rho(x)\rho(y)}}\right\}, ~~|\nabla_x \operatorname{P}(x,z)|\lesssim \frac{1}{|x-z|^N},
  \end{align*}
  for any $(x,y,z)\in\Omega\times\Omega\times\partial\Omega$, see \cite{BiVi}.
 Since $|\rho(x)-\rho(y)|\leq |x-y|$ we have 
\begin{align*}
\max\left\{\rho(x)\rho(y),|x-y|^2\right\}\asymp\max\left\{|x-y|,\rho(x),\rho(y)\right\}^2.
\end{align*} 
 Thus, 
 \begin{align}\label{051220145}
 \min\left\{1,\left(\frac{|x-y|}{\sqrt{\rho(x)\rho(y)}}\right)^\gamma\right\}
 \asymp \frac{|x-y|^\gamma}{\left(\max\left\{|x-y|,\rho(x),\rho(y)\right\}\right)^{\gamma}}~~\text{ for }~\gamma>0.
 \end{align}  
 Therefore, 
\begin{align}\label{1311201427a}
\operatorname{G}(x,y)\asymp  \rho(x)\rho(y)\mathbf{N}_{2,2}(x,y),~~\operatorname{P}(x,z)\asymp \rho(x)\mathbf{N}_{\alpha,\alpha}(x,z)
\end{align}   
and 
\begin{align}\label{1311201426a}
|\nabla_x \operatorname{G}(x,y)|\lesssim \rho(y)\mathbf{N}_{1,1}(x,y), ~~~|\nabla_x \operatorname{P}(x,z)|\lesssim \mathbf{N}_{\alpha,\alpha}(x,z)
\end{align}for all $(x,y,z)\in \overline{\Omega}\times\overline{\Omega}\times\partial\Omega,$  $ \alpha\geq 0.$ \medskip\\
\begin{proof}[Proof of Theorem \ref{121120142} and Theorem \ref{121120147}] By \eqref{1311201427a}, the following equivalence holds  
 \begin{align*}
 &\mathbf{G}\left[\left(\mathbf{P}[\sigma]\right)^q\right]\lesssim \mathbf{P}[\sigma]<\infty ~~a.e \text{ in }~ \Omega. 
\Longleftrightarrow 
 &\mathbf{N}_{2,2}\left[\left(\mathbf{N}_{2,2}[\sigma]\right)^q\rho^{q+1}\right]\lesssim \mathbf{N}_{2,2}[\sigma]<\infty ~~a.e \text{ in }~ \Omega.
 \end{align*}
Furthermore
 \begin{align*}
 U\asymp \mathbf{G}[U^q]+\mathbf{P}[\sigma] &\Longleftrightarrow U\asymp \rho\mathbf{N}_{2,2}[\rho U^q]+\rho\mathbf{N}_{2,2}[\sigma],
 \end{align*}
 which in turn is equivalent to
 $$V\asymp \mathbf{N}_{2,2}[\rho^{q+1}V^q]+\mathbf{N}_{2,2}[\sigma]\text{ with }V=U\rho^{-1}.
 $$
 By Proposition \ref{2010201417} and \ref{2010201414} we have:
\begin{align*}
 \operatorname{Cap}_{I_{\frac{2}{q}},q'}(K)\asymp\operatorname{Cap}^{q+1}_{\mathbf{N}_{2,2},q'}(K\times\{0\})\qquad\forall\,K\subset \mathbb{R}^{N-1}, K\text{ compact}.
 \end{align*}
 if $\Omega=\mathbb{R}^{N}_+$,  and
 \begin{align*}
  \operatorname{Cap}^{\partial\Omega}_{\frac{2}{q},q'}(K)\asymp\operatorname{Cap}^{q+1}_{\mathbf{N}_{2,2},q'}(K)\qquad\forall \,K\subset \partial\Omega, K\text{ compact}.
  \end{align*}
   if $\Omega$ is a bounded domain. 
 Thanks to Theorem \eqref{2010201420} with $\omega=\sigma$, $\alpha=2,\beta=2,\alpha_0=q+1$ and proposition \ref{2010201434}, we get the results.    

\end{proof}\medskip\\
\begin{proof}[Proof of Theorem \ref{1311201437} and \ref{1311201440}] 
By \eqref{1311201427a} and \eqref{1311201426a}, we have 
\begin{align}\label{211020142}
&\operatorname{G}(x,y)\leq C\rho(x)\rho(y)\mathbf{N}_{1,1}(x,y),~~|\nabla_x\operatorname{G}(x,y)|\leq C\rho(y)\mathbf{N}_{1,1}(x,y),\\
&\operatorname{P}(x,z)\leq C\rho(y)\mathbf{N}_{1,1}(x,z),~~|\nabla_x\operatorname{P}(x,z)|\leq C\mathbf{N}_{1,1}(x,z),\label{211020143}
\end{align}
for all $(x,y,z)\in\Omega\times\Omega\times\partial\Omega$ and for some constant $C>0$.\\
For any $u\in W_{loc}^{1,1}(\Omega)$, we set
\begin{align*}
\mathbf{F}(u)(x)=\int_{\Omega}\operatorname{G}(x,y)|u(y)|^{q_1-1}u(y)|\nabla u(y)|^{q_2}dy+\int_{\partial\Omega}\operatorname{P}(x,z)d\sigma(z).
\end{align*}  
Using  \eqref{211020142} and \eqref{211020143}, we have 
\begin{align*}
&|\mathbf{F}(u)|\leq C\rho(.) \mathbf{N}_{1,1}\left[|u|^{q_1}|\nabla u|^{q_2}\rho(.)\right]+C\rho(.)\mathbf{N}_{1,1}[|\sigma|],\\&
|\nabla \mathbf{F}(u)|\leq C \mathbf{N}_{1,1}\left[|u|^{q_1}|\nabla u|^{q_2}\rho(.)\right]+C\mathbf{N}_{1,1}[|\sigma|].
\end{align*} 
Therefore, we can easily  see that if \begin{align}\label{2010201448}
\mathbf{N}_{1,1}\left[\left(\mathbf{N}_{1,1}[|\sigma|]\right)^{q_1+q_2}(\rho(.))^{q_1+1}\right]\leq \frac{\left(q_1+q_2-1\right)^{q_1+q_2-1}}{\left(C(q_1+q_2)\right)^{q_1+q_2}} \mathbf{N}_{1,1}[|\sigma|]<\infty ~~a.e \text{ in }~~ \Omega
\end{align} 
 holds, then $\mathbf{F}$ is the map from $\mathbf{E}$ to $\mathbf{E}$, where 
 $$\mathbf{E}=\left\{u\in W_{loc}^{1,1}(\Omega): |u|\leq \lambda\rho(.)\mathbf{N}_{1,1}[|\sigma|],~ |\nabla u|\leq \lambda\mathbf{N}_{1,1}[|\sigma|]~~\text{ a.e in }~\Omega\right\}$$ 
 with $\lambda=\frac{C(q_1+q_2)}{q_1+q_2-1}$. \\
Assume that \eqref{2010201448} holds.
 We denote $\cal S$ by  the subspace of functions  $f\in W_{loc}^{1,1}(\Omega)$ with norm 
 \begin{align*}
 ||f||_{{\cal S}}=||f||_{L^{q_1+q_2}(\Omega,(\rho(.))^{1-q_2}dx)}+|||\nabla f|||_{L^{q_1+q_2}(\Omega,(\rho(.))^{1+q_2}dx)}<\infty.
 \end{align*}
 Clearly, $\mathbf{E}\subset \cal S$,   $\mathbf{E}$ is closed under the strong topology of $\cal{S}$ and convex. \\
 On the other hand, 
 it is not difficult to show that  $\mathbf{F}$ is continuous and $\mathbf{F}(\mathbf{E})$ is precompact in $\cal{S}$.
 Consequently,  by  Schauder's fixed point theorem, there exists $u\in \mathbf{E}$ such that  $\mathbf{F}(u)=u$. Hence, $u$ is a solution of \eqref{1311201439}-\eqref{2010201447} and it satisfies 
 \begin{align*}
 |u|\leq \lambda\rho(.)\mathbf{N}_{1,1}[|\sigma|],~ |\nabla u|\leq \lambda\mathbf{N}_{1,1}[|\sigma].
 \end{align*} 
 Thanks to Theorem \ref{2010201420}  and Proposition \ref{2010201417}, \ref{2010201414}, we verify that assumptions \eqref{1311201438} and \eqref{2010201447} in Theorem \ref{1311201437} and \ref{1311201440} are equivalent to \eqref{2010201448}. 
This completes the proof of the Theorems. \end{proof} 
\begin{remark}\label{opt} We do not know whether  conditions \eqref{1311201438} and \eqref{2010201446} are
optimal or not. It is noticeable that if  $\mathbf{P}[|\sigma|]\in L^{q_1+q_2}(\Omega,\rho^{1-q_2} dx)$, it is proved in \cite[Th 1.1] {MV2} that, if $\Omega$ is a ball, then $|\sigma|$ belongs to the Besov-Sobolev space $B^{-\frac{2-q_2}{q_1+q_2},q_1+q_2}(\partial\Omega)$. Therefore  inequality
 \begin{align*}
 |\sigma|(K)\leq C \left(\operatorname{Cap}^{\partial\Omega}_{\frac{2-q_2}{q_1+q_2},(q_1+q_2)'}(K)\right)^{\frac{1}{(q_1+q_2)'}}\end{align*}
 holds for any Borel set $K\subset\partial\Omega$, and it is a necessary condition for \eqref{2010201446} to hold since $\frac{1}{(q_1+q_2)'}<1$. In a general $C^2$ bounded domain, it is easy to see that this property, proved in a particular case in \cite[Th 2.2] {MV1} is still valid thanks  to the equivalence relation (2.23) therein between Poisson's kernels, see also the proof of Proposition \ref{2010201414}. The difficulty for obtaining a necessary condition of existence lies in the fact that, if the inequality $u\geq \mathbf{P}[\sigma]$ is clear, $|\nabla u|\gtrsim \rho^{-1}\mathbf{P}[\sigma]$ is not true.
It can also be shown that if 
 $$|u|^{q_1}|\nabla u|^{q_2}\leq C (\mathbf{G}(|\sigma|))^{q_1}(\rho\mathbf{N}_{1,1}[|\sigma|])^{q_2}\in L^1(\Omega,\rho(.) dx),$$  
then $\sigma $   is absolutely continuous with respect to  $ \operatorname{Cap}^{\partial\Omega}_{\frac{2-q_2}{q_1+q_2},(q_1+q_2)'}$.
\end{remark}
\section{Extension to Schr\"odinger operators with Hardy potentials}
We can apply Theorem \ref{2010201420}  to solve the problem 
\begin{equation*}\begin{array}{lll}
-\Delta u -\frac{\kappa}{\rho^2}u= u^q~~&\text{in}~\Omega, \\
\phantom{-\Delta u -\frac{\kappa}{\rho^2}}
u = \sigma\quad ~&\text{on }~\partial\Omega,\\ 
\end{array}\end{equation*} 
where $\kappa\in [0,\frac{1}{4}]$ and $\sigma\in\mathfrak{M}^+(\partial\Omega)$.\smallskip

Let $\operatorname{G}_\kappa,\operatorname{P}_\kappa$ be the Green kernel and Poisson kernel  of $-\Delta-\frac{\kappa}{\rho^2}$ in $\Omega$ with $\kappa\in [0,\frac{1}{4}]$. It is proved that 
\begin{align*}
&\operatorname{G}_{\kappa}(x,y)\asymp  \min\left\{\frac{1}{|x-y|^{N-2}}, \frac{(\rho(x)\rho(y))^{\frac{1+\sqrt{1-4\kappa}}{2}}}{|x-y|^{N-1+\sqrt{1-4\kappa}}}\right\},\\&
\operatorname{P}_{\kappa}(x,z)\asymp \frac{(\rho(x))^{\frac{1+\sqrt{1-4\kappa}}{2}}}{|x-z|^{N-1+\sqrt{1-4\kappa}}},
\end{align*}
 for all $(x,y,z)\in \overline{\Omega}\times\overline{\Omega}\times\partial\Omega$, see \cite{FMT,MT1,GV}. 
Therefore, from \eqref{051220145} we get 
\begin{align}\label{1311201427}
&\operatorname{G}_{\kappa}(x,y)\asymp  (\rho(x)\rho(y))^{\frac{1+\sqrt{1-4\kappa}}{2}}\mathbf{N}_{1+\sqrt{1-4\kappa},2}(x,y),\\&\label{1311201428}
\operatorname{P}_\kappa(x,z)\asymp (\rho(x))^{\frac{1+\sqrt{1-4\kappa}}{2}}\mathbf{N}_{\alpha,1-\sqrt{1-4\kappa}+\alpha}(x,z),
\end{align}    for all $(x,y,z)\in \overline{\Omega}\times\overline{\Omega}\times\partial\Omega,$  $ \alpha\geq 0.$ 
We denote 
\begin{align*}
\mathbf{P}_\kappa[\sigma](x)=\int_{\partial\Omega}\operatorname{P}_\kappa(x,z)d\sigma(z), ~~\mathbf{G}_\kappa[f](x)=\int_{\Omega}\operatorname{G}_\kappa(x,y)f(y)dy
\end{align*}
for any $\sigma\in \mathfrak{M}^+(\partial\Omega), f\in L^1(\Omega,\rho^{\frac{1+\sqrt{1-4\kappa}}{2}}dx), f\geq 0$. Then  the unique  weak solution of 
$$\begin{array}{lll}
-\Delta u-\frac{\kappa}{\rho^2}u=f\qquad&\text {in }\Omega,\\
\phantom{ -\Delta u-\frac{\kappa}{\rho^2}}
u=\sigma\qquad&\text {on }\partial\Omega,
\end{array}$$ 
satisfies the following integral equation \cite{GV} 
\begin{align*}
u=\mathbf{G}_\kappa[f]+\mathbf{P}_\kappa[\sigma]~~\text{a.e. in } \Omega.
\end{align*}
 As in the proofs of Theorem \ref{121120142} and Theorem \ref{121120147} the relation
\begin{align*}
\mathbf{G}_\kappa\left[\left(\mathbf{P}_\kappa[\sigma]\right)^q\right]\lesssim \mathbf{P}_\kappa[\sigma]<\infty ~~\text{a.e  in }~~ \Omega, 
\end{align*}
is equivalent to 
\begin{align*}
\mathbf{N}_{1+\sqrt{1-4\kappa},2}\left[\left(\mathbf{N}_{1+\sqrt{1-4\kappa},2}[\sigma]\right)^q\rho^{\frac{(q+1)(1+\sqrt{1-4\kappa})}{2}}\right]\lesssim \mathbf{N}_{1+\sqrt{1-4\kappa},2}[\sigma]<\infty ~~\text{a.e  in }~~ \Omega,
\end{align*}
and the relation
\begin{align*}
U\asymp \mathbf{G}_\kappa[U^q]+\mathbf{P}_\kappa[\sigma],
\end{align*} 
is equivalent to 
\begin{align*}V\asymp \mathbf{N}_{1+\sqrt{1-4\kappa},2}[\rho^{\frac{(q+1)(1+\sqrt{1-4\kappa})}{2}}V^q]+\mathbf{N}_{1+\sqrt{1-4\kappa},2}[\sigma]\quad\text{with } V=U\rho^{-\frac{1+\sqrt{1-4\kappa}}{2}}.
\end{align*}
Thanks to Theorem \ref{2010201420} with $\omega=\sigma$,  $\alpha=1+\sqrt{1-4\kappa},\beta=2,\alpha_0=\frac{(q+1)(1+\sqrt{1-4\kappa})}{2}$ and proposition \ref{2010201434}, \ref{2010201417}, \ref{2010201414}, we  obtain. 
\begin{theorem} \label{1311201410}Let $q>1, 0\leq \kappa\leq \frac{1}{4}$ and $\sigma\in \mathfrak{M}^+(\partial\Omega)$.  Then,  the following statements are equivalent
\begin{description}
  \item${\bf 1}$ There exists $C>0$ such that the following inequalities hold  
\begin{align}\label{2010201428}
\sigma(O)\leq C \operatorname{Cap}_{I_{\frac{q+3-(q-1)\sqrt{1-4\kappa}}{2q}},q'}(O)\end{align}
for any Borel set $O\subset\mathbb{R}^{N-1}$ if $\Omega=\mathbb{R}^{N}_+$  and 
\begin{align}\label{2010201429}
  \sigma(O)\leq C \operatorname{Cap}^{\partial\Omega}_{\frac{q+3-(q-1)\sqrt{1-4\kappa}}{2q},q'}(O)\end{align}
for any Borel set $O\subset\partial\Omega$ if  $\Omega$ is a bounded domain.    
   \item${\bf 2}$ There exists $C>0$ such that the inequality
 \begin{align}
 \label{2010201433b}
& \mathbf{G}_\kappa\left[\left(\mathbf{P}_\kappa[\sigma]\right)^q\right]\leq C \mathbf{P}_\kappa[\sigma]<\infty ~~a.e \text{ in }~~ \Omega,
 \end{align}
 holds.
    
   \item[3.] Problem 
    \begin{equation}\label{2010201431} \begin{array}{lll}
 -\Delta u -\frac{\kappa}{\rho^2}u= u^q~~&\text{in}~\Omega, \\
 \phantom{ -\Delta u -\frac{\kappa}{\rho^2}}
 u = \varepsilon\sigma\quad ~&\text{on }~\partial\Omega,\\ 
 \end{array} \end{equation}
   has a positive solution for $\varepsilon>0$ small enough.
   \end{description}
   Moreover, there is a constant $C_0>0$ such that if any one
   of the two statements ${\bf 1}$ and ${\bf 2}$  holds with $C\leq C_0$, then equation \ref{2010201431} 
   has a solution u with $\varepsilon=1$ which satisfies 
   \begin{align}
   u\asymp \mathbf{P}_\kappa[\sigma].
   \end{align}
   Conversely, if \eqref{2010201431} has
   a solution u with $\varepsilon=1$, then the two statements ${\bf 1}$ and ${\bf 2}$ hold for some $C>0$.    
 \end{theorem}
 
 \begin{remark} The problem \eqref{2010201431} admits a subcritical range
 $$1<q< \frac{N+\frac{1+\sqrt{1-4\kappa}}{2}}{N+\frac{1+\sqrt{1-4\kappa}}{2}-2}.
 $$
 If the above inequality, the problem can be solved with any positive measure provided $\sigma(\partial\Omega)$ is small enough. 
 The role of this critical exponent has been pointed out in \cite{MT1} and \cite{GV} for the removability of boundary isolated singularities of solutions of 
 $$ -\Delta u -\frac{\kappa}{\rho^2}u+ u^q=0~\text{in}~\Omega
 $$
 i.e. solutions which vanish on the boundary except at one point. Furthermore the complete study of the problem
     \begin{equation}\label{2010201431+1} \begin{array}{lll}
 -\Delta u -\frac{\kappa}{\rho^2}u+ u^q=0\quad&\text{in}~\Omega, \\
 \phantom{ -\Delta u -\frac{\kappa}{\rho^2}+ u^q}
 u = \sigma\quad ~&\text{on }~\partial\Omega,\\ 
 \end{array} \end{equation}
 is performed in \cite{GV} in the supercritical range
  $$q\geq \frac{N+\frac{1+\sqrt{1-4\kappa}}{2}}{N+\frac{1+\sqrt{1-4\kappa}}{2}-2}.
 $$
 The necessary and sufficient condition is therein expressed in terms of the absolute continuity of $\sigma$ with respect to the $\operatorname{Cap}_{I_{\frac{q+3-(q-1)\sqrt{1-4\kappa}}{2q}},q'}$-capacity.
 \end{remark}

\end{document}